\documentclass[12pt,leqno]{amsart}
\usepackage{verbatim,amssymb,amsfonts,latexsym,amsmath,amsthm,tikz}
\usetikzlibrary{arrows}

\newtheorem{theorem}{Theorem}[section]
\newtheorem{lemma}[theorem]{Lemma}
\newtheorem{corollary}[theorem]{Corollary}
\newtheorem{proposition}[theorem]{Proposition}
\newtheorem{question}[theorem]{Question}
\theoremstyle{definition}

\theoremstyle{remark}
\newtheorem{remark}[theorem]{Remark}
\newtheorem*{claim}{Claim}
\newtheorem*{sub-claim}{sub-claim}

\newcommand{\N}{\mathbb{N}}

\newcommand{\F}{\mathcal{F}}
\newcommand{\G}{\mathcal{G}}
\newcommand{\explicitSet}[1]{\left\lbrace #1 \right\rbrace}
\newcommand{\brackets}[1]{\left\langle #1 \right\rangle}
\newcommand{\set}[2]{\explicitSet{#1 \colon #2}}
\newcommand{\seq}[2]{\brackets{#1 \colon #2}}

\newcommand{\0}{\emptyset}
\renewcommand{\a}{\alpha}
\renewcommand{\b}{\beta}

\renewcommand{\k}{\kappa}
\newcommand{\s}{\sigma}

\newcommand{\w}{\omega}
\newcommand{\sub}{\subseteq}
\newcommand{\rest}{\!\restriction\!}

\newcommand{\closure}[1]{\overline{#1}}

\newcommand{\B}{\mathcal{B}}

\newcommand{\plim}{p\mbox{-}\!\lim_{n \in \w}}

\newcommand{\card}[1]{\left\lvert #1 \right\rvert}

\newcommand{\p}{\mathbb{P}}

\newcommand{\pwmf}{\mathcal{P}(\w)/\mathrm{fin}}
\newcommand{\continuum}{\mathfrak{c}}
\newcommand{\pseudo}{\mathfrak{p}}
\newcommand{\tower}{\mathfrak{t}}
\newcommand{\dom}{\mathfrak d}

\renewcommand{\S}{\mathcal S}

\renewcommand{\]}{]\!]}
\newcommand{\thick}{\Theta}

\begin{document}

\title{$G_\delta$ semifilters and $\w^*$}
\author{Will Brian}
\address {
William R. Brian\\
Department of Mathematics\\
Tulane University\\
6823 St. Charles Ave.\\
New Orleans, LA 70118}
\email{wbrian.math@gmail.com}
\author{Jonathan Verner}
\address {
Jonathan L. Verner\\
Department of Logic, Faculty of Arts \\
Charles University \\
Palachovo n\'am. 2 \\
116 38 Praha 1 \\
Czech Republic}
\email{jonathan.verner@ff.cuni.cz}
\subjclass[2010]{Primary: 03E17, 54D35. Secondary: 22A15, 03E35, 06A07}
\keywords{semifilter, $\w^*$, (weak) $P$-filter, (weak) $P$-set, minimal ideal, minimal/maximal idempotent}

\thanks{Work of the second author was partially supported by the joint FWF-GA\v{C}R grant no. I 1921-N25, The continuum, forcing, and large cardinals and a 
postdoctoral fellowship at the Faculty of Arts, Charles University.}

\maketitle

\begin{abstract}
The ultrafilters on the partial order $([\w]^{\w},\sub^*)$ are the free ultrafilters on $\w$, which constitute the space $\w^*$, the Stone-\v{C}ech remainder of $\w$. If $U$ is an upperset of this partial order (i.e., a \emph{semifilter}), then the ultrafilters on $U$ correspond to closed subsets of $\w^*$ via Stone duality.

If, in addition, $U$ is sufficiently ``simple'' (more precisely, $G_\delta$ as a subset of $2^\w$), we show that $U$ is similar to $[\w]^{\w}$ in several ways. First, $\pseudo_U = \tower_U = \pseudo$ (this extends a result of Malliaris and Shelah). Second, if $\dom = \continuum$ then there are ultrafilters on $U$ that are also $P$-filters (this extends a result of Ketonen). Third, there are ultrafilters on $U$ that are weak $P$-filters (this extends a result of Kunen).

By choosing appropriate $U$, these similarity theorems find applications in dynamics, algebra, and combinatorics. Most notably, we will prove that $(\w^*,+)$ contains minimal left ideals that are also weak $P$-sets.
\end{abstract}

\section{Introduction}

The main theme of this paper is that there is a class of ``simple'' semifilters whose members all look essentially like $[\w]^{\w}$, the set of infinite subsets of $\w$. We will prove several theorems along these lines, and also find applications of these theorems.

Recall that any semifilter is naturally identified with a subset of $2^\w$ via characteristic functions. The ``simple'' class of semifilters we are interested in are those that are $G_\delta$ in $2^\w$ (where $2^\w$ has its standard topology as the Cantor set). Some of our proofs will also work for co-meager semifilters. Clearly $[\w]^{\w}$ is in this class because it is co-countable. We will show that many of its properties, including some that correspond to interesting properties of $\w^*$, can be proved for any other semifilter in this class as well.

For example, in Section~\ref{sec:p&t}, we show that $G_\delta$ semifilters all satisfy the Malliaris-Shelah equality $\pseudo = \tower$ (see \cite{M&S}). That is, defining $\pseudo_\G$ and $\tower_\G$ appropriately, we show that for a $G_\delta$ semifilter $\G$, $\pseudo_\G = \tower_\G = \pseudo$.

In Section~\ref{sec:d=c}, we show that if $\dom = \continuum$ then every $G_\delta$ semifilter admits an ultrafilter that is also a $P$-filter. This generalizes a result of Kentonen from \cite{Ket}, which says the same thing for $[\w]^{\w}$.

In Section~\ref{sec:weakpsets}, we show that every $G_\delta$ semifilter admits an ultrafilter that is a weak $P$-filter. This generalizes a result of Kunen from \cite{Kun}, which says the same thing for $[\w]^{\w}$.

In Section~\ref{sec:applications} we have collected a few applications of these results. Among other things, we show that $(\w^*,+)$ contains a minimal left ideal that is also a weak $P$-set. Any such ideal is prime, and the idempotents it contains are both minimal and left-maximal. This strengthens a result of Zelenyuk from \cite{YZ2}.

\section{Preliminaries}\label{sec:prelims}

A \emph{semifilter} on $\w$ is a subset $\S$ of $\mathcal P(\w)$ such that $\0 \neq \S \neq \mathcal P (\w)$ and $\S$ is closed upwards in $\sub^*$ (as usual, $A \sub^* B$ means $A \setminus B$ is finite). We think of semifilters as partial orders, naturally ordered by $\sub^*$. The largest possible semifilter is $[\w]^{\w}$, the set of all infinite subsets of $\w$.

A partial order is \emph{antisymmetric} if $a \leq b$ and $b \leq a$ implies $a=b$; some authors even include this in the definition of a partial order. We note that our partial orders do not enjoy this property. However, each one has an \emph{antisymmetric quotient}, namely the set of equivalence classes of the form $[X] = \set{Y \sub \w}{X \sub^* Y \sub^* X}$. For example, the antisymmetric quotient of $[\w]^{\w}$ is the familiar order $\pwmf$ (without the bottom element). In what follows, we have no need (and no desire) to work with equivalence classes, and will not need to use the antisymmetry axiom anywhere. Therefore we choose to work with subsets of $\w$ rather than equivalence classes thereof. It is worth pointing out, though, that all of our proofs and constructions ``factor through'' the antisymmetric quotient, and can be interpreted as results about $\pwmf$ and its uppersets.

If $\S$ is a semifilter, then a \emph{filter} on $\S$ is a filter on the partial order $(\S,\sub^*)$. Specifically, $\F \sub \S$ is a filter on $\S$ whenever
\begin{itemize}
\item $\F \neq \0$.
\item $A \in \F$ and $A \sub^* B$ implies $B \in \F$.
\item $A,B \in \F$ implies $A \cap B \in \F$.
\end{itemize}
An \emph{ultrafilter} on $\S$ is a maximal filter on $\S$. $\B \sub \S$ is a \emph{filter base} on $\S$ if $\set{A \sub \N}{B \sub^*A \text{ for some } B \in \B}$ is a filter on $\S$. A set is \emph{centered} in $\S$ if it is contained in some filter base.

The collection of all ultrafilters on $[\w]^{\w}$ is denoted $\w^*$. This set has a natural topology as the Stone-\v{C}ech remainder of $\w$, with basic open sets of the form $A^* = \set{\F \in \w^*}{A \in \F}$. Every filter $\F$ on $[\w]^{\w}$ corresponds to a closed subset of $\w^*$, namely $\hat \F = \bigcap_{A \in \F}A^*$. $\hat \F$ is called the \emph{Stone dual} of $\F$. For more on Stone duality and the topology of $\w^*$, we refer the reader to \cite{JvM}.

If $\F$ is a filter on some semifilter $\S$, then $\F$ is also a filter on $[\w]^{\w}$, although an ultrafilter on $\S$ may not be an ultrafilter on $[\w]^{\w}$. Thus the (ultra)filters on a semifilter $\S$ correspond to closed subsets of $\w^*$. For certain choices of $\S$, these closed sets may have interesting algebraic/dynamical/combinatorial properties, and for certain choices of the ultrafilter they may also have interesting topological properties. The interplay between these two choices will give rise to our applications in Section~\ref{sec:applications}.

A subset $X$ of $\w^*$ is a $P$\emph{-set} if, whenever $\seq{U_n}{n < \w}$ is a sequence of open sets each of which contains $X$, $X$ is in the interior of $\bigcap_{n < \w}U_n$. $X$ is a \emph{weak $P$-set} if the closure of each countable $D \sub \w^* \setminus X$ is disjoint from $X$. $\F$ is a \emph{(weak) $P$-filter} iff $\hat \F$ is a (weak) $P$-set.

The basic open neighborhoods of $2^\w$ are of the form
$$\[A \rest F\] = \set{X \in 2^\w}{X \cap F = A \cap F}$$
for $A \sub \w$ and finite $F \sub \w$. If $s \sub [0,n]$, we will write $\[s\]$ for $\[s \rest [0,n]\]$.

We mention here a special semifilter that will appear in several places throughout this paper. A set $A \sub \w$ is \emph{thick} if $A$ contains arbitrarily long intervals, and we let $\thick$ denote the semifilter of thick sets. The ultrafilters on $\thick$ correspond (via Stone duality) precisely to the minimal left ideals of $(\w^*,+)$ (see Lemma 3.2 in \cite{Brn}). It was this observation that first motivated the study of ultrafilters on $\thick$, and this in turn motivated our work here.

We end this section by mentioning some results on the descriptive complexity of
semifilters. Recall that a set has the \emph{Baire property} if it differs from
an open set by a meager set. All Borel sets as well as analytic and co-analytic sets have the
Baire property. The following proposition (stated for semifilters in \cite{BaZd})
shows that definable semifilters are either very small or very large:

\begin{proposition}\label{prop:bairemeager}
If a semifilter has the Baire property then it is either meager or co-meager.
\end{proposition}

Meager (or co-meager) filters have a very convenient characterization 
due to Talagrand and, independently, Jalaili-Naini (see \cite{Tal}, \cite{J-N}).
It was noticed in \cite{BaZd} that it applies to semifilters as well.
Given two semifilters \(\mathcal S,\mathcal G\) we say that 
\(\mathcal S\) is \emph{Rudin-Blass above} \(\mathcal G\) (\(\mathcal S\geq_{RB}\mathcal G\)) if there is a finite-to-one function \(f:\omega\to\omega\)
such that \(A\in\mathcal G\) if and only if \(f^{-1}[A]\in\mathcal S\) (this is the standard Rudin-Blass ordering extended to semifilters).

\begin{proposition}\label{prop:talagrand}
A semifilter is co-meager iff it is Rudin-Blass above \([\w]^{\w}\). It is meager iff it is Rudin-Blass above the Fr\'echet filter.
\end{proposition}

As noted in the introduction, some of our results below will hold for co-meager semifilters. These are related to the $G_\delta$ semifilters in the following natural way:

\begin{corollary}\label{cor:comeagers}
A semifilter $\S$ is co-meager if and only if it contains a $G_\delta$ semifilter.
\end{corollary}
\begin{proof}
Because semifilters are closed under making finite modifications, every semifilter is dense in $2^\w$. The ``if'' direction follows. For the ``only if'' direction, let $\S$ be a co-meager semifilter and, using the first part of Proposition~\ref{prop:talagrand}, let $f: \w \to \w$ be a finite-to-one function such that \(f^{-1}[A]\in\mathcal S\) for any infinite $A \sub \w$.

For each $n$, let
$$U_n = \set{X \in 2^\w}{\exists \text{ distinct } m_1,\dots,m_n \text{ with } \textstyle \bigcup_{1 \leq k \leq n}f^{-1}(m_k) \sub X}.$$
$U_n$ is open and closed upwards with respect to $\sub$. Therefore $\G = \bigcap_{n \in \w}U_n$ is $G_\delta$, and is easily seen to be closed upwards with respect to $\sub^*$. In other words, $\G$ is a $G_\delta$ semifilter, and $\G \sub \S$ by construction.
\end{proof}

Finally, at the first level of the Borel hierarchy, we have a characterization
of \(G_\delta\) semifilters somewhat reminiscent of Mazur's characterization
of \(F_\sigma\) ideals (see \cite{Maz}). Recall that a \emph{monotone lower semicontinuous functional} on \(\mathcal P(\omega)\) is a function 
\(g:\mathcal P(\omega)\to\mathbb R^+_0\cup\{\infty\}\) satisfying:
\begin{enumerate}
  \item  \(g(\emptyset)=0\) and \(g(A)<\infty\) for each finite \(A\);
  \item  \(g(A)\leq g(B)\) for each \(A\subseteq B\); and
  \item  \(g(A)=\sup\{g(A\cap n):n<\omega\}\).
\end{enumerate}

\begin{proposition}\label{prop:functional}
A semifilter  \(\mathcal S\) is \(G_\delta\) 
iff there is a monotone lower semicontinuous functional \(g\) such that \(\mathcal S=\{X:g(X)=\infty\}\).
\end{proposition}
\begin{proof}
The set  \(\{X:g(X)=\infty\}\) is clearly \(G_\delta\), so we need only proof the other direction. 
The fastest way to see this is to use the following lemma of Mazur (\cite{Maz}, Proposition 1.1.):

\begin{claim}
An  \(F_\sigma\) family \(\mathcal I\) of subsets of \(\omega\) closed under taking subsets can be written
as an increasing union of closed sets which are themselves closed under taking subsets.
\end{claim}

\begin{proof}[Proof of Claim]
Write  \(\mathcal I=\bigcup_{n<\omega} F_n\) with \(F_n\subseteq F_{n+1}\) for each \(n<\omega\). Then let 
$$\overline{F}_n = \{X\cap Y:Y\in F_n, X\in\mathcal P(\omega)\}.$$
It is clear that \(\mathcal I=\bigcup_{n<\omega} \overline{F}_n\). Moreover each \(\overline{F}_n\) is closed
since it is the image of a compact set (\(F_n\times\mathcal P(\omega)\)) under a continuous map (\(\cap\)).
\renewcommand{\qedsymbol}{$\blacksquare$}
\end{proof}

Using the above claim (and De Morgan laws), write  \(\mathcal S\) as an intersection of a decreasing sequence
of open sets closed under taking supersets, $\mathcal S=\bigcap_{n<\omega} U_n$, and define \(g\) by $g(A) = \sup\{n: A\in G_n\}$.
\end{proof}

\section{$\pseudo$ and $\tower$}\label{sec:p&t}

In this section, we show that, in a certain combinatorial sense, $G_\delta$ semifilters have the same ``depth'' as $[\w]^{\w}$. To make this precise, we define two cardinal invariants that measure the ``depth'' of a partial order (similar definitions appear in \cite{Brn}). To avoid trivialities, we assume in this section that every partial order $\p$ has a non-atomic antisymmetric quotient: i.e., for every $a \in \p$ there is some $x \in \p$ with $a \not\leq x \leq a$.

For a partial order $\p$,
\begin{itemize}
\item $\pseudo_\p$ is the smallest size of an unbounded centered subset of $\p$.
\item $\tower_\p$ is the smallest size of an unbounded chain in $\p$.
\end{itemize}
Note that unbounded centered sets and chains must exist in $\p$ because of our requirement that the separative quotient of $\p$ is non-atomic. In fact, our condition on $\p$ is equivalent to the condition that every centered set in $\p$ is contained in an unbounded centered set.

It is easily checked that $\pseudo_{\pwmf} = \pseudo_{[\w]^{\w}} = \pseudo$ and $\tower_{\pwmf} = \tower_{[\w]^{\w}} = \tower$. In other words, our notation is justified, and these cardinal characteristics naturally extend the familiar $\pseudo$ and $\tower$. Note that, in the definition of $\tower_\p$, it suffices to consider (reverse) well-ordered chains.

A deep new result of Malliaris and Shelah is that $\pseudo = \tower$ (see \cite{M&S}). That is, these two notions of ``depth'' coincide for $\pwmf$. The main result of this section (Theorem~\ref{thm:M&S} below) asserts that $G_\delta$ semifilters also satisfy the Malliaris-Shelah equality, and moreover have the same ``depth'' as $\pwmf$.

\begin{proposition}\label{prop:RBinequalities}
Given any semifilters $\S$ and $\G$, if \(\mathcal S\geq_{RB}\mathcal G\) then 
\(\mathfrak p_{\mathcal S}\leq\mathfrak p_{\mathcal G}\) and 
\(\mathfrak t_{\mathcal S}\leq\mathfrak t_{\mathcal G}\).
\end{proposition}
\begin{proof} We only show the first inequality, the proof of the second is analogous.
It is sufficient to find, for each centered family \(\mathcal F\subseteq\mathcal G\)
with no lower bound in \(\mathcal G\), a centered family \(\mathcal H\subseteq\mathcal S\) of the same size having no lower bound in \(\mathcal S\). 

Let \(f:\omega\to\omega\) be a finite-to-one function witnessing that \(\mathcal S\geq_{RB}\mathcal G\) and let \(\mathcal H=\{f^{-1}[A]:A\in \mathcal F\}\). Notice
that, since \(f\) is finite-to-one, if \(A\subseteq^*B\) then \(f^{-1}[A]\subseteq^* f^{-1}[B]\). It follows that \(\mathcal H\) is centered. Aiming towards a contradiction,
assume that \(A\) is a lower bound for \(\mathcal H\) in \(\mathcal S\). Since
\(f\) is a Rudin-Blass reduction of \(\mathcal S\) to \(\mathcal G\) it follows that
\(B=f[A]\in\mathcal G\). Since \(\mathcal F\) does not have a lower bound in 
\(\mathcal G\) there must be an \(F\in\mathcal F\) such that \(B\setminus F\) is
infinite. Then \(f^{-1}[B]\cap A\) is also infinite and disjoint from \(f^{-1}[F]\in\mathcal H\), contradicting the assumption that \(A\) was a lower for
\(\mathcal H\).
\end{proof}

\begin{corollary}\label{cor:towerinequality}
If $\S$ is a co-meager semifilter then $\pseudo_\S \leq \tower_\S \leq \tower$.
\end{corollary}
\begin{proof}
Since every chain is centered in $\S$, it is clear from the definitions of $\pseudo_\S$ and $\tower_\S$ that $\pseudo_\S \leq \tower_\S$ (and this does not depend on $\S$ being co-meager). If $\S$ is co-meager, then by Proposition~\ref{prop:talagrand} $\S \geq_{RB} [\w]^\w$, and by Proposition~\ref{prop:RBinequalities} $\tower_\S \leq \tower_{[\w]^\w} = \tower$.
\end{proof}

\begin{theorem}\label{thm:M&S}
If $\G$ is a $G_\delta$ semifilter, then $\pseudo_\G = \tower_\G = \pseudo$.
\end{theorem}
\begin{proof}
By the aforementioned result of Malliaris and Shelah, $\pseudo = \tower$. Using Corollaries \ref{cor:comeagers} and \ref{cor:towerinequality}, $\pseudo_\G \leq \tower_\G \leq \tower$. Therefore it is sufficient to prove $\pseudo \leq \pseudo_\G$.

Let $U_n$, $n < \w$, be open sets such that $\G = \bigcap_{n < \w}U_n$. Replacing $U_n$ with $\bigcap_{m \leq n}U_m$ if necessary, we may assume that the $U_n$ are decreasing.

Given $\k < \pseudo$, we want to show $\k < \pseudo_\G$. Let $\set{A_\a}{\a < \k}$ be centered in $\G$. By Bell's Theorem (see \cite{Bel}), it suffices to use $\mathrm{MA}_{\s\text{-centered}}^\k$ to find a lower bound for this family in $\G$.

To do this, we use a common variant of the Mathias forcing. Specifically, we have a forcing notion $\p$ whose conditions are pairs $(s,F)$, where $s$ is a finite subset of $\w$ and $F$ is a finite subset of $\k$. We say that $(s,F) \leq (t,G)$ if and only if $t \supseteq s$, $G \supseteq F$, and $t \setminus s \sub \bigcap_{\a \in F}A_\a$. Intuitively, the condition $(s,F)$ promises that $s$ will be contained in the set $X$ we are trying to build, and $X \setminus s$ will be contained in each $A_\a$, $\a \in F$.

For each $\a < \k$,
$$D_\a = \set{(s,F)}{\a \in F}$$
is dense in $\p$ because $(s,F \cup \{\a\})$ always extends $(s,F)$. For each $n < \w$,
$$E_n = \set{(s,F)}{\[s\] \sub U_n}$$
is also dense in $\p$. To see that $E_n$ is dense, fix any $(s,F)$. Let $A_0 = \bigcap_{\a \in F}A_\a$ and let $A = A_0 \cup s$. Since $A_0 \in \G$, $A \in \G$ and therefore there is some $n$ such that $\[A \rest [0,n)\] \sub U_n$. In this case, we have $(A \cap n,F) \leq (s,F)$ and $(A \cap n,F) \in E_n$.

By $\mathrm{MA}_{\s\text{-centered}}^\k$, there is a filter $G$ on $\p$ meeting all the $D_\a$ and all the $E_n$. Let $A = \bigcup \set{s}{(s,F) \in G}$. It is straightforward to check that $A \in U_n$ for every $n$ (because $G \cap E_n \neq \0$) and that $A \sub^* A_\a$ for every $\a$ (because $G \cap D_\a \neq \0$). Therefore $A \in \G$ and $A$ is a lower bound for $\set{A_\a}{\a < \k}$. Hence $\k < \pseudo_\G$, and it follows that $\pseudo \leq \pseudo_\G$.
\end{proof}

\begin{remark}\label{rem:1}
The requirement that $\G$ be $G_\delta$ in Theorem~\ref{thm:M&S} cannot be relaxed to include general $F_\s$ semifilters. To see this, let $\S$ denote the semifilter of syndetic sets: these are sets with ``bounded gaps'', i.e., $A$ is syndetic iff its complement fails to be thick. It is a straightforward exercise to show that $\S$ is $F_\s$ in $2^\w$. Also, one can show that $\S$ contains no lower bound for the sequence $\seq{\set{m \cdot 2^n}{m \in \w}}{n \in \w}$, even though each element of this sequence is in $\S$. This shows $\pseudo_\S = \tower_\S = \aleph_0$. 
\end{remark}

\begin{remark}\label{rem:2}
The requirement that $\G$ be $G_\delta$ in Theorem~\ref{thm:M&S} cannot be relaxed to include general co-meager semifilters. To see this, let $\S$ denote the semifilter of sets that are either syndetic or thick. Since the syndetic sets form an $F_\s$ semifilter and the thick sets form a $G_\delta$ semifilter, $\S$ is $\Delta^0_3$ and co-meager. Once again, however, $\S$ contains no lower bound for the sequence $\seq{\set{m \cdot 2^n}{m \in \w}}{n \in \w}$, and $\pseudo_\S = \tower_\S = \aleph_0$. 
\end{remark}

The preceding remarks show that the conclusion $\pseudo = \pseudo_\G = \tower_\G$ fails as soon as we relax the requirement that $\G$ be $G_\delta$. Jan Star\'y has asked us whether the conclusion $\pseudo_\G = \tower_\G$ can also fail. Unfortunately, we must leave this question open:

\begin{question}\label{q:stary}
Is there a semifilter $\S$ such that $\pseudo_\S \neq \tower_\S$?
\end{question}

By the following simple observation, however, we can say that the answer to Star\'y's question is consistently negative:

\begin{proposition}
\emph{CH} implies that $\pseudo_\S = \tower_\S$ for every semifilter $\S$.
\end{proposition}
\begin{proof}
Suppose CH holds and let $\S$ be any semifilter. Clearly $\pseudo_\S \leq \tower_\S$, so if $\pseudo_\S$ is uncountable then $\pseudo_\S = \tower_\S$. The following claim (which does not require CH) completes the proof.

\begin{claim}
If $\pseudo_\S \leq \aleph_0$ then $\pseudo_\S = \tower_\S = \aleph_0$.
\end{claim}

To prove the claim, first note that $\pseudo_\S$ cannot be finite: if $F$ is finite and centered in $\S$, then $F$ has a lower bound in $\S$, namely $\bigcap F$. So suppose $\pseudo_\S = \aleph_0$ and let $\set{A_n}{n \in \w}$ be an unbounded centered subset of $\S$. Then $\set{\bigcap_{m \leq n}A_m}{n \in \w}$ is an unbounded chain, so that $\tower_\S = \aleph_0$.
\end{proof}

In Remarks \ref{rem:1} and \ref{rem:2}, we see that it is very easy to find semifilters $\S$ with $\pseudo_\S, \tower_\S < \pseudo$. It is clearly consistent to have $\tower_\S \leq \tower$ for every semifilter (e.g., this follows from $\tower = \continuum$). The following proposition shows that the opposite is also consistent. It also shows that Corollary~\ref{cor:towerinequality} cannot be extended to all semifilters with the Baire property.

\begin{proposition}
It is consistent that there is a meager semifilter $\S$ with $\pseudo < \pseudo_\S$.
\end{proposition}
\begin{proof}
Recall that $\mathrm{add}(\mathcal M)$ denotes the smallest number of meager sets whose union is non-meager. Suppose that:
\begin{enumerate}
\item $\mathrm{add}(\mathcal M) = \continuum = \aleph_3$.
\item $\pseudo = \tower = \aleph_1$.
\item there is a sequence $\seq{A_\a}{\a < \w_2}$ in $[\w]^\w$ such that $A_\a \sub^* A_\b$ and $A_\a \neq^* A_\b$ whenever $\a < \b < \w_2$.
\end{enumerate}
Such a model is obtained, for example, by a length-$\w_3$, finite-support iteration of the Hechler forcing (a.k.a., dominating forcing) over a model of GCH. It is well-known (see, e.g., Section 11.6 of \cite{AB2}) that this model satisfies $\continuum = \mathrm{add}(\mathcal M) = \aleph_3$ and $\tower = \aleph_1$. Our third requirement is true in this model by an application of Theorem 4.1 of \cite{B&D}. According to this theorem, after $\w_2$ steps of our iteration, there will be a sequence $\mathcal A$ satisfying the conditions of $(3)$. However, it is easy to see that these conditions are absolute (provided $\w_2$ is absolute, which it is here). So the same sequence $\mathcal A$ will make $(3)$ true in the final model.

We now work within a model satisfying $(1) - (3)$ to reach the desired conclusion.

Let $\seq{A_\a}{\a < \w_2}$ be a sequence in $[\w]^\w$ such that $A_\a \sub^* A_\b$ and $A_\a \neq^* A_\b$ whenever $\a < \b < \w_2$. Let $\F$ be the (semi)filter generated by this sequence; explicitly,
$$\F = \bigcup_{\a \in \w_2}\set{X \in 2^\w}{A_\a \sub^* X}.$$
For any fixed infinite set $A$, $\set{X \in 2^\w}{A \sub^* X}$ is meager in $2^\w$. Since $\mathrm{add}(\mathcal M) = \continuum > \aleph_2$, $\F$ is meager.

To finish the proof, it suffices to show $\pseudo_\F = \aleph_2$. Clearly $\set{A_\a}{\a \in \w_2}$ is an unbounded chain in $\F$, so $\pseudo_\F \leq \tower_\F \leq \aleph_2$. For the opposite inequality, let $\set{B_\b}{\b \in \w_1}$ be centered in $\F$. By the definition of $\F$, for each $\b$, there is some $\a_\b$ such that $B_\b \sub^* A_{\a_\b}$. Letting $\a = \sup \set{a_\b}{\b \in \w_1}$, $A_\a$ is a bound for $\set{B_\b}{\b \in \w_1}$ in $\F$.
\end{proof}

\section{$P$-filters from $\dom = \continuum$}\label{sec:d=c}

If $f,g \in \w^\w$, we say $g$ \emph{dominates} $f$, and write $f \leq^* g$, whenever $\set{n \in \w}{f(n) \geq g(n)}$ is finite. The \emph{dominating number} $\dom$ is the smallest size of some $D \sub \w^\w$ such that every $f \in \w^\w$ is dominated by some $g \in D$.

In this section we show that if $\dom = \continuum$ then every $G_\delta$ semifilter admits a $P$-ultrafilter.

\begin{lemma}\label{lem:fastfunction}
Let $\S$ be a semifilter, and let $U$ be an open subset of $2^\w$ with $\S \sub U$. For each $X \in \S$, there is function $f_X: \w \to \w$ such that, for every $m \in \w$, $\[X \rest [m,f_X(m))\] \sub U$.
\end{lemma}
\begin{proof}
Fix $m \in \w$ and $X \in \S$. Let $\set{M_i}{i < 2^m}$ enumerate the subsets of $m$. For each $i < 2^m$, let $X_i = M_i \cup (X \setminus m)$ and let $n_i$ be the least natural number satisfying $\[X_i \rest n_i\] \sub U$. $X_i$ is a finite modification of $X \in \S$, so $X_i \in \S \sub U$; therefore some such $n_i$ must exist. Let $f_X(m) = \max \set{n_i}{i < 2^m}$. If $Y \in \[X \rest [m,f_X(m))\]$, then for some $i$ we have $Y \cap m = M_i$, which gives $Y \in \[X_i \rest f_X(m)\] \sub \[X_i \rest n_i\] \sub U$.
\end{proof}

The following lemma generalizes Proposition 1.3 in \cite{Ket}; see also Proposition 6.24 in \cite{AB2}.

\begin{lemma}\label{lem:extension}
Let $\G$ be a $G_\delta$ semifilter. Suppose $\set{A_n}{n \in \w}$ is a decreasing sequence in $\G$; also, suppose $\mathcal B \sub \G$, $|\mathcal B| < \dom$, and for each $B \in \mathcal B$, $\{B\} \cup \set{A_n}{n \in \w}$ is centered in $\G$. Then $\set{A_n}{n \in \w}$ has a bound $A \in \G$ such that, for any $B \in \mathcal B$, $A$ and $B$ have a common bound in $\G$.
\end{lemma}
\begin{proof}
Replacing $A_n$ with $\bigcap_{m \leq n}A_m$ if necessary, we may assume that the $A_n$ are decreasing. This only changes $A_n$ finitely, and does not affect the other hypotheses or the conclusion of the lemma. Let $\seq{U_n}{n < \w}$, be a sequence of open subsets of $2^\w$ with $\G = \bigcap_{n \in \w}U_n$. Replacing $U_n$ with $\bigcap_{m \leq n}U_m$ if necessary, we may assume that the $U_n$ are also decreasing.

For each $B \in \mathcal B$, let $g_B(n) = f_{A_n \cap B}(n)$, where $f_{A_n \cap B}$ is the function described in Lemma~\ref{lem:fastfunction}. This is well-defined since $A_n \cap B \in \G$ (because $\set{A_n}{n \in \w} \cup \{B\}$ is centered in $\G$). As $|\mathcal B| < \dom$, there is some $h \in \w^\w$ that is not dominated by any $g_B$.

Let $A = \bigcup_{n \in \w}(A_n \cap h(n))$. We claim that this $A$ satisfies the conclusions of the lemma. There are two things to check: that $A$ is a bound for $\set{A_n}{n \in \w}$, and that $A$ and $B$ have a common bound in $\G$ for any $B \in \mathcal B$.

Because the $A_n$ are decreasing, $A \setminus A_n \sub \bigcup_{m < n}(A_m \cap h(m))$, which is a finite set. Thus $A \sub^* A_n$ for every $n$ and $A$ is a bound for $\set{A_n}{n \in \w}$.

It remains to show that for any $B \in \B$, $A$ and $B$ have a common lower bound in $\G$.

Since $h$ is not dominated by $g_B$, there is an infinite $C \sub \w$ such that $g_B(n) < h(n)$ for all $n \in C$. By induction, find an infinite increasing sequence $\seq{n_i}{i \in \w}$ of elements of $C$ such that, for each $i$, $h(n_i) < n_{i+1}$. Then put $\tilde A = \bigcup_{i \in \w}([n_i,h(n_i)) \cap B \cap A_{n_i})$. By our requirements on the $n_i$, the intervals $[n_i,h(n_i))$ are disjoint. Clearly, $\tilde A \sub A \cap B$, and it remains to show $\tilde A \in \G$. For any $i$, we have $f_{A_{n_i} \cap B}(n_i) = g_B(n_i) < h(n_i)$, so that $\tilde A \cap [n_i,f_{A_{n_i} \cap B}(n_i)) = A_{n_i} \cap B \cap [n_i,f_{A_{n_i} \cap B}(n_i))$. By the definition of the function $f_{A_{n_i} \cap B}$, this means $\tilde A \sub U_{n_i}$. This is true for all the $n_i$, so $\tilde A$ is in infinitely many of the $U_m$. As the $U_m$ are decreasing, $\tilde A \in \bigcap_{m \in \w}U_m = \G$.
\end{proof}

The following theorem extends to $G_\delta$ semifilters a classical result of Ketonen about $[\w]^{\w}$ (see Theorem 1.2 and Proposition 1.4 in \cite{Ket} or Theorem 9.25 in \cite{AB2}). It also extends Theorem 3.5 of \cite{Brn}, where similar conclusions are reached for the semifilter of thick sets using the hypothesis $\pseudo = \continuum$.

\begin{theorem}\label{thm:d=c}
Let $\G$ be a $G_\delta$ semifilter. If $\dom = \continuum$, then there is a $P$-ultrafilter on $\G$. In fact,
\begin{enumerate}
\item If $\dom = \continuum$, then every filter on $\G$ that is generated by fewer than $\continuum$ sets is included in some $P$-ultrafilter on $\G$.
\item Every ultrafilter on $\G$ that is generated by fewer than $\dom$ sets is a $P$-ultrafilter.
\end{enumerate}
\end{theorem}
\begin{proof}
For $(1)$, let $\F_0$ be a basis for a filter on $\G$ with $\card{\F_0} < \dom$ and let $\set{\seq{X_n^\a}{n < \w}}{\a \in \continuum}$ be an enumeration of all countable decreasing sequences in $\G$. To avoid trivialities, we assume $\F_0$ contains the Fr\'echet filter. We construct by recursion an increasing sequence of filter bases $\seq{\F_\a}{\a < \continuum}$, such that $\card{\F_\a} = \aleph_0 \cdot \card{\a}$.

For limit $\a$, we put $\F_\a = \bigcup_{\b < \a}\F_\b$. Given that $\F_\a$ has already been constructed, we obtain $\F_{\a+1}$ as follows. If there is some $n < \w$ such that $\F_\a \cup \{X_n^\a\}$ is not centered in $\G$, then we put $\F_{\a+1} = \F_\a$. If this is not the case, $\F_\a \cup \set{X_n^\a}{n \in \w}$ is centered in $\G$. Since $\card{\F_\a} < \dom$, we may apply Lemma~\ref{lem:extension} to find a lower bound $X^\a_\w$ for $\seq{X_n^\a}{n < \w}$ such that, for any $B \in \F_\a$, we have $X_\w^\a \cap B \in \G$. In particular, $\F_\a \cup \set{X_\w^\a \cap B}{B \in \F_\a}$ is a filter base, and we define this to be $\F_{\a+1}$. Clearly $\card{\F_{\a+1}} = \aleph_0 \cdot \card{\F_\a}$, and this completes the recursive construction. Let $\F$ be the filter generated by $\bigcup_{\a < \continuum}\F_\a$.

We must prove that $\F$ is a $P$-ultrafilter on $\G$. It is obvious that $\bigcup_{\a < \continuum}\F_\a$ is a filter basis in $\G$, so $\F$ is a filter in $\G$. To see that $\F$ is an ultrafilter in $\G$, let $A \in \G$. There is some $\a < \continuum$ such that $\seq{X^\a_n}{n < \w}$ is the constant sequence $X_n^\a = A$. If $\{A\} \cup \F$ is centered, so is $\{A\} \cup \F_\a$, and at step $\a$ of our construction we found a set $X_\w^\a \sub^* A$ and put $X_\w^\a \in \F_{\a+1}$. This implies $A \in \F$ (recall that $\F$ contains the Fr\'echet filter). Thus, for every $A \in \G$, either $A \in \F$ or $\{A\} \cup \F$ is not centered in $\G$. Hence $\F$ is an ultrafilter on $\G$.

To see that $\F$ is a $P$-filter, let $\seq{A_n}{n < \w}$ be a decreasing sequence of elements of $\F$. For some $\a$, we have $X_n^\a = A_n$ for all $n$. At stage $\a$ of our construction, we added some set $X_\w^\a$ to $\F$ that is a lower bound for this sequence.

For $(2)$, let $\F$ be an ultrafilter on $\G$, and let $\B$ be a basis for $\F$ such that $\card{\B} < \dom$. If $\seq{A_n}{n < \w}$ is a decreasing sequence in $\F$, then $\set{A_n}{n \in \w} \cup \B$ is centered. By Lemma~\ref{lem:extension}, there is some $A_\w \in \G$ that is a lower bound for $\set{A_n}{n < \w}$ and that, for every $B \in \B$, satisfies $A_\w \cap B \in \G$. Then $\{A_\w\} \cup \B$ is centered in $\G$. Since $\B$ is a basis for the ultrafilter $\F$ on $\G$, this implies $A_\w \in \F$. Thus an arbitrary decreasing sequence of elements of $\F$ has a lower bound in $\F$; i.e., $\F$ is a $P$-filter.
\end{proof}

We end this section by observing that Theorem~\ref{thm:d=c} does not hold for all co-meager semifilters.

\begin{proposition}
There is a co-meager semifilter $\S$ such that there cannot be a $P$-ultrafilter on $\S$.
\end{proposition}
\begin{proof}
The semifilter in question is the semifilter $\S$ of IP sets (which is defined in Section~\ref{sec:applications}). We will show in the proof of Corollary~\ref{cor:IPsets} that $\S$ is co-meager, and we will show in Lemma~\ref{lem:glasner} that every ultrafilter on $\S$ is also an ultrafilter on $[\w]^{<\w}$. The current proposition now follows from the well-known fact that no $P$-point in $\w^*$ can consist entirely of IP sets (this is implied, for example, by Corollary 8.37 in \cite{H&S}).
\end{proof}

\section{Weak $P$-filters from ZFC}\label{sec:weakpsets}

In this section we show that there are weak P-ultrafilters on any co-meager
semifilter, in particular on the semifilter of thick sets. This will then
be used in Section~\ref{sec:applications}. 

To construct an ultrafilter with nice combinatorial properties, one usually 
uses a recursive construction. We first divide the combinatorial property into
a large set of requirements and then at step \(\alpha\) of the construction 
the filter constructed so far is extended in such a way that 
\(\alpha\)-th requirement is met. There are two main problems. 

The first is that the recursive construction might stop before all of the
requirements are met. To avoid this problem, one can use an idea going back
to Posp\'i\v{s}il (\cite{Pop}): start with an independent system and make sure 
that at each step a large enough part of this system remains independent modulo
the filter constructed so far. This guarantees, in particular, that the
filter is not an ultrafilter. Since there are independent systems of size 
\(\mathfrak c\) this typically allows one to take care of \(\mathfrak c\)-many
requirements. 

The second problem is coming up with the requirements.
For example, to get a weak P-point, the natural requirement would be given
by a single countable sequence of ultrafilters and it would require that the
constructed ultrafilter is not in the closure of this sequence. Unfortunately,
there are far too many countable sequences of ultrafilters --- we would need
to meet \(2^{\mathfrak c}\)-many requirements, but our construction only has
\(\mathfrak c\)-many steps. To overcome \emph{this} problem Kunen 
(\cite{Kun}) did something counterintuitive: he replaced the easier
problem of constructing a weak P-point by a harder problem of 
constructing \(\mathfrak c\)-O.K. points. The clever part was that the 
combinatorial property of being a \(\mathfrak c\)-O.K. point, while decidedly 
uglier, can actually be divided into \(\mathfrak c\)-many requirements
thus leaving hope for our recursive construction. 

We modify Kunen's proof and show that it allows us to construct 
weak P-ultrafilters on semifilters. First we need Kunen's definition (see \cite{Kun})
of an O.K.-set: A (closed) subset  \(X\subseteq\omega^*\) is \(\kappa\)-O.K. if for each sequence \(\langle U_n:n<\omega\rangle\) of open
neigbourhoods of \(X\) there is a family \(\{V_\gamma:\gamma<\kappa\}\) of neigbourhoods of \(X\) such that for each finite
\(\Gamma\in[\kappa]^{<\omega}\) the following is true:
$$ \bigcap_{\gamma\in \Gamma} V_\gamma\subseteq U_{|\Gamma|} $$

Note that if $\k \leq \lambda$, then every $\lambda$-O.K. set is also $\k$-O.K. The following lemma, also due to Kunen, shows that $\k$-O.K. sets (with $\k$ uncountable) are weak $P$-sets: 

\begin{lemma}\label{lem:okweak}
A closed  \(\omega_1\)-O.K. set is a weak P-set.
\end{lemma}

\begin{proof}
Let  \(F\) be a closed \(\omega_1\)-O.K. set and \(D=\{p_n:n<\omega\}\) a countable set disjoint from \(F\). Fix a descending
sequence of neigbourhoods \(\langle U_n:n<\omega\rangle\) of \(F\) such that \(p_n\) is not contained in any \(U_m\) 
for \(m>n\). For this sequence, choose open neigbourhoods \(\{V_\alpha:\alpha<\omega_1\}\) of \(F\) witnessing that it 
is \(\omega_1\)-O.K. It is easy to see that \(p_n\) can only be an element of at most \(n+1\) many \(V_\alpha\)s, so we can fix
\(\alpha_n<\omega_1\) such that \(p_n\) is not contained in any \(V_\beta\) for \(\beta\geq\alpha_n\). Let
\(\beta=\sup\{\alpha_n:n<\omega\}<\omega_1\). Then \(V_\beta\) is a neigbourhood of \(F\) disjoint from \(D\).
\end{proof}

\begin{remark}
Later van Mill (\cite{vM2}) generalized this lemma to show that a closed 
\(\omega_1\)-O.K. subset of \(X^*=\beta X\setminus X\) for any locally
compact \(\sigma\)-compact \(X\) is actually disjoint even from the closure of any
\emph{ccc} subset of \(X^*\) disjoint from it.
\end{remark}

\begin{lemma}
A filter  \(\mathcal F\) on \(\omega\) is \(2^\omega\)-O.K. if for each sequence \(\langle F_n:n<\omega\rangle\) of elements of
\(\mathcal F\) there are \(\{V_\gamma:\gamma<2^\omega\}\subseteq\mathcal F\) such that for each \(n<\omega\) and
\(\gamma_1,\ldots,\gamma_n<2^\omega\) the set
$$ V_{\gamma_1}\cap \cdots\cap V_{\gamma_n}\setminus F_n $$
is finite.
\end{lemma}

\begin{theorem} \label{thm:comeagerOKset} 
If \(\mathcal S\) is a co-meager semifilter then there is an ultrafilter 
on \(\mathcal S\) which is a \(2^\omega\)-O.K. set (and hence a weak $P$-set).
\end{theorem}

The proof uses large independent linked systems. 
For our purposes we will slightly modify the relevant defition.

Given a filter \(\mathcal F\) and a semifilter \(\mathcal S\) we say that a family of sets
\(\mathcal A(C,R) = \{X^\beta_{\alpha,n}:\alpha\in C,n<\omega,\beta\in R\}\) is a \(C\times R\) independent linked 
matrix modulo \((\mathcal F,\mathcal S)\) if 

\begin{enumerate}
  \item For each \(\alpha\in C,\beta\in R\) the sequence \(\langle X^\beta_{\alpha,n}:n<\omega\rangle\) is increasing in
        \(\subseteq\).
  \item For each  \(F\in \mathcal F\), each finite set of rows \(R_0\in[R]^{<\omega}\), each choice of natural numbers
        \(N:R_0\to\omega\) and each choice of sets of columns \(C_0:R_0\to[A]^{<\omega}\) such that \(|C_0(\beta)|\leq N(\beta)\)
        for each \(\beta\in R_0\) the intersection
        $$    F\cap\bigcap_{\beta\in R_0}\bigcap_{\alpha\in C_0(\beta)} X^\beta_{\alpha,N(\beta)}$$
        is in \(\mathcal S\).
  \item For each row  \(\beta\in R\) and any \(n+1\)-many columns \(C_0\in[C]^{n+1}\) the intersection
        $$    \bigcap_{\alpha\in C_0} X^\beta_{\alpha,n}$$
        is finite.
\end{enumerate}

If  \(\mathcal A(C,R)\) is an independent matrix and \(R_0\subseteq R\) are some rows, we will use
\(\mathcal A(C,R\setminus R_0)\) to denote the matrix constructed from \(\mathcal A(C,R)\) by deleting the rows (with indices)
in \(R_0\).

\begin{lemma}[Kunen]
There is a  \(2^\omega\times 2^\omega\)-independent linked matrix modulo \(({\mathcal Fr},[\omega]^{\omega})\).
\end{lemma}

Kunen used an elaborate recursive construction using trees. The following simple
proof is due to P. Simon

\begin{proof}[Proof of the Fact]
We shall construct such a family consisting of subsets of the countable set  
$S=\{(k,f):k\in\omega,f\in{}^{\mathcal P(k)}\mathcal{PP}(k)\}$. Given $A,B\subseteq\omega$ and $n<\omega$ let 
$$   X_{A,n}^B=\{(k,f)\in S:|f(B\cap k)|\leq n\ \&\ A\cap k\in f(B\cap k)\}.$$
It is routine, if perhaps somewhat involved, to check that $\{X_{A,n}^B: n<\omega, A,B\subseteq\omega\}$ 
is a $2^\omega\times2^\omega$ independent linked family modulo \(({\mathcal Fr},[\omega]^{\omega})\).
\end{proof}

\begin{corollary}\label{cor:comeagrindependent}
If  \(\mathcal S\) co-meager, then there is \(2^\omega\times 2^\omega\)-independent linked
matrix modulo \(({\mathcal Fr},\mathcal S)\)
\end{corollary}

\begin{proof}
By Proposition~\ref{prop:talagrand}, there is a finite-to-one function \(f:\omega\to\omega\) such that 
for each \(X\subseteq\omega\) we have that \(X\in [\omega]^{\omega} \iff f^{-1}[X]\in\mathcal S\). 
Let \(\{X^\beta_{\alpha,n}:\alpha,\beta\in2^\omega,n<\omega\}\) be the matrix given by the previous fact.
Let \(Y^\beta_{\alpha,n}=f^{-1}[X^\beta_{\alpha,n}]\). It is easy to check that the \(Y\)'s form the required independent
matrix.
\end{proof}

We are now ready to prove the main Theorem~\ref{thm:comeagerOKset}. The proof constructs the maximal filter by a recursive construction which is kept going by a large independent linked matrix. 

There will be two kinds of requirements that we will need to meet. One type
takes care of a single countable sequence of neigbourhoods that potentially 
comes into play in the definition of \(2^\omega\)-O.K. sets, the other
type will take care of a single set in our semifilter to guaratee that the resulting 
filter is maximal. 

The following two lemmas say that the requirements can be met at the cost
of sacrificing at most countably many rows from our matrix. In both of these
lemmas \(\mathcal S\) is some co-meager semifilter. 

\begin{lemma}
Let  \(\mathcal A(C,R)\) be an independent linked matrix modulo \((\mathcal F,{\mathcal S})\) and \(\langle Y_n:n<\omega\rangle\)
a sequence of elements of \(\mathcal F\). Fix any row \(\beta\in R\). Then there are \(\{V_\gamma:\gamma<2^\omega\}\) such
that \(\mathcal A(C,R\setminus\{\beta\})\) is independent linked modulo \((\mathcal F\cup\{V_\alpha:\alpha<2^\omega\},{\mathcal S})\)
and, moreover, for any finite set of indices \(\Gamma\in[2^\omega]^{<\omega}\) the set
$$ \bigcap_{\alpha\in\Gamma} V_\alpha\setminus Y_n$$
is finite.
\end{lemma}

\begin{proof}
Just let 
$$ V_\gamma = \bigcup_{n<\omega} Y_n\cap A^\beta_{\gamma,n}.$$
\end{proof}

\begin{lemma} \label{lem:maximalstep} 
Let  \(\mathcal A(C,R)\) independent linked matrix modulo \((\mathcal F,{\mathcal S})\) and \(X\in{\mathcal S}\). Then there is 
a finite number of rows \(R^\prime\) and an extension \(\mathcal F^\prime\) of \(\mathcal F\) such that 
\(\mathcal A(C,R\setminus R^\prime)\) is an independent linked matrix modulo \(\mathcal F^\prime,{\mathcal S}\) and either \(X\) or 
\(\omega\setminus X\) are in \(\mathcal F^\prime\) or there is \(F\in\mathcal F^\prime\) such that both \(F\cap X\) and
\(F\setminus X\) are not in \(\mathcal S\).
\end{lemma}

\begin{proof}
If  \(\mathcal A(C,R)\) is not independent modulo \((\mathcal F\cup\{X\},{\mathcal S})\) then there are an \(F_0\in\mathcal F\);
a finite set of rows \(R_0\); sizes \(N_0:R_0\to\omega\), and finite sets of columns \(C_0:R_0\to 2^\omega\) each of size 
given by \(N_0\) such that
$$   X\cap F_0\cap\bigcap_{\beta\in R_0}\bigcap_{\alpha\in C_0(\beta)} X^\beta_{\alpha,N_0(\beta)}$$
is not in \(\mathcal S\). If \(\mathcal A(C,R\setminus R_0)\) independent modulo \((\mathcal F\cup\{\omega\setminus X\},{\mathcal S})\) then
there are an \(F_1\in\mathcal F\); a finite set of rows \(R_1\subseteq R\setminus R_0\); sizes \(N_1:R_1\to\omega\); and finite
sets of columns \(C_1:C_1\to 2^\omega\) each of size given by \(N_1\) such that
$$   (\omega\setminus X)\cap F_1 \cap\bigcap_{\beta\in R_1}\bigcap_{\alpha\in C_1(\beta)} X^\beta_{\alpha,N_1(\beta)}$$
is not in \(\mathcal S\). Then let 
$$Z = F_0\cap
      \bigcap_{\beta\in R_0}\bigcap_{\alpha\in C_0(\beta)} X^\beta_{\alpha,N_0(\beta)}\cap
      F_1\cap
      \bigcap_{\beta\in R_1}\bigcap_{\alpha\in C_1(\beta)} X^\beta_{\alpha,N_1(\beta)}.$$
and \(R^\prime=R_0\cup R_1\) By construction \(Z\cap X\) and \(Z\setminus X\) are not in \(\mathcal S\) and 
\(\mathcal A(C,R\setminus R^\prime)\) is an independent linked matrix modulo \((\mathcal F\cup\{Z\},{\mathcal S})\).
\end{proof}

The proof of the theorem is now a routine recursive construction. 

\begin{proof}[Proof of Theorem \ref{thm:comeagerOKset}]
Let  \(\mathcal A(2^\omega,2^\omega)\) be an independent linked matrix modulo \(({\mathcal Fr},{\mathcal S})\) given by 
Corollary \ref{cor:comeagrindependent} (notice that, by Proposition \ref{prop:talagrand}, \(\mathcal S\) is Rudin-Blass above \([\omega]^\omega\)).

Enumerate \(\mathcal S\) as \(\{X_\alpha:\alpha\in 2^\omega\}\) and all countable 
sequences of sets from \(\mathcal S\) as \(\{\langle Y_{\alpha,n}:n<\omega\rangle:\alpha<2^\omega\}\). Using the two lemmas we shall
recursively construct a sequence of filters \(\langle \mathcal F_\alpha:\alpha<2^\omega\rangle\) putting the used rows
into \(R_\alpha\) along the way so that the following conditions are satisfied:

\begin{enumerate}
  \item  \(|R_\alpha|\leq\omega\cdot|\alpha|\) for each \(\alpha<2^\omega\);
  \item  \(\mathcal F_\alpha\subseteq F_\beta\) and \(R_\alpha\subseteq R_\beta\) for each \(\alpha<\beta<2^\omega\);
  \item  \(\mathcal A(2^\omega,2^\omega\setminus R_\alpha)\) is an independent linked matrix modulo 
         \((\mathcal F_\alpha,{\mathcal S})\) for each \(\alpha<2^\omega\);
  \item if the sequence  \(Y_\alpha\) is contained in \(\mathcal F_\alpha\) then there are
        \(\{V_\gamma:\gamma<2^\omega\}\subseteq\mathcal F_{\alpha+1}\) such that for each finite set of indices
        \(\Gamma\in[2^\omega]^{<\omega}\) the intersection \(\bigcap_{\gamma\in \Gamma} V_\gamma\setminus Y_{\alpha,|I|}\)
        is finite; and
  \item either \(X_\alpha\in\mathcal F_{\alpha+1}\) or \(\omega\setminus X_\alpha\in\mathcal F_{\alpha+1}\) or there is a
        set \(F\in\mathcal F_{\alpha+1}\) such that both \(F\cap X_\alpha\not\in{\mathcal S}\)
        and \(F\cap (\omega\setminus X_\alpha)\not\in{\mathcal S}\).
\end{enumerate}

We start by letting  \(\mathcal F_0={\mathcal Fr}\) and \(R_0=\emptyset\). At limit stages we take unions and at successor stages we
use the previous lemmas to guarantee conditions 3 and 4. Finally, we let
$$\mathcal F = \bigcup_{\alpha<2^\omega}\mathcal F_\alpha.$$
Condition 5 guarantees that \(\mathcal F\) is an ultrafilter on \(\mathcal S\) 
while condition 4 guarantees that \(\mathcal F\) is an $2^\w$-O.K. set. This finishes the proof of the theorem.
\end{proof}

\begin{remark} If \(\mathcal S\) is a co-ideal (a semifilter with the Ramsey property) then, using the previous  construction, we can in fact build a full ultrafilter
(i.e. an ultrafilter on \([\omega]^{\w}\)) consisting of sets in \(\mathcal S\)
which is an O.K.-point (and hence a weak P-point). The key here is that in lemma
\ref{lem:maximalstep} we can guarantee that the third option will not take place (since either X will be in \(S\) or its complement will). This will be discussed further in the following section.
\end{remark}


\section{Applications to algebra and combinatorics}\label{sec:applications}

Prior to this section, we have focused our attention on examining the structure of certain semifilters and building certain kinds of ultrafilters in them. In this section, we explore the question of what such ultrafilters might be used for.

First we look at applications to the standard dynamical/algebraic structure of $\w^*$. For $p \in \w^*$, recall that $\s(p)$ is the unique ultrafilter generated by $\set{A+1}{A \in p}$. Equivalently, $\s$ is the restriction to $\w^*$ of the unique continuous extension to $\b\w$ of the successor map on $\w$. This function $\s$, called the \emph{shift map}, provides the canonical dynamical structure for $\w^*$.

Related to the shift map on $\w^*$ is the canonical semigroup structure on $\w^*$. Addition in $\w^*$ is defined by putting  $p+q = \plim \s^n(q)$ for every $p,q \in \w^*$. Here, ``$\plim \s^n(q)$'' has its usual meaning: $r = \plim \s^n(q)$ if and only if, for every neighborhood $U$ of $r$, $\set{n \in \w}{\s^n(q) \in U} \in p$. 

Recall that, if $(X,f)$ is a dynamical system, then $Y \sub X$ is a \emph{minimal subsystem} if $Y$ is closed under $f$ and closed topologically, and if no proper nonempty subset of $Y$ also has these properties. If $(X,+)$ is a semigroup, then $Y \sub X$ is a \emph{minimal left ideal} if $X+Y = Y$, and no proper nonempty subset of $Y$ has this property.

These core notions from dynamics and algebra are, for $\w^*$, related to each other and to the notion of ultrafilters on semifilters by the following result from \cite{Brn}.

\begin{lemma}\label{lem:thicksets}
Let $\F$ be any filter on $[\w]^{\w}$. The following are equivalent:
\begin{enumerate}
\item $\F$ is an ultrafilter on $\thick$.
\item $\hat \F$ is a minimal dynamical subsystem of $(\w^*,\s)$.
\item $\hat \F$ is a minimal left ideal of $(\w^*,+)$.
\end{enumerate}
\end{lemma}
\begin{proof}
The equivalence of $(2)$ and $(3)$ is well known, and a proof can be found, e.g., in \cite{Brg}. For the equivalence of these and $(1)$, see Lemma 3.2 in \cite{Brn}.
\end{proof}

Recall that an \emph{idempotent ultrafilter} is any $p \in \w^*$ such that $p+p = p$. The idempotents of $\w^*$ (indeed, any semigroup) admit a natural partial order as follows: $p$ and $q$ are idempotent, then $p \leq q$ if and only if $p+q = q+p = p$. If $q+p = p$ (but not necessarily $p+q = p$), we write $p \leq_L q$.

It is known (see, e.g., \cite{H&S}, Theorem 1.38) that an idempotent $p$ is minimal with respect to $\leq$ if and only if $p$ belongs to some minimal left ideal. Such ultrafilters are called \textbf{minimal idempotents}. It is also known that if $q$ is any idempotent then there is a minimal idempotent $p$ with $p \leq q$.

For many years it was a stubborn open question whether $\w^*$ contains any $\leq_L$-maximal idempotents, and a good deal of work was done on this question (see Questions 9.25 and 9.26 in \cite{H&S}, Questions 5.5(2),(3) in \cite{HSZ}, Problems 4.6 and 4.7 in \cite{Brn}, and \cite{Zln}). In \cite{YZ2}, Zelenyuk finally answered this question in the affirmative. The following application of Theorem~\ref{thm:comeagerOKset} provides an alternative proof of Zelenyuk's result, and strengthens the result by showing that some minimal left ideal is a weak $P$-set.

\begin{theorem}\label{thm:algebra}
There is a minimal left ideal of $\w^*$ that is also a weak $P$-set. It follows that
\begin{enumerate}
\item There is an idempotent ultrafilter that is both minimal and $\leq_L$-maximal.
\item There is a minimal left ideal $L \sub \w^*$ such that, for any $p,q \in \w^*$, $p+q \in L$ if and only if $q \in L$.
\item The minimal left ideals are not homeomorphically embedded in $\w^*$.
\end{enumerate}
\end{theorem}

For $(3)$, recall that $Y, Z \sub X$ are \emph{homeomorphically embedded} in $X$ if there is some homeomorphism $h: X \to X$ such that $h(Y) = Z$. It is well-known that the minimal left ideals of $\w^*$ are all homeomorphic, and in fact the homeomorphisms between them arise naturally from the algebraic structure (they are shifts of each other). This result says that the minimal left ideals are nonetheless topologically distinguishable, and the natural homeomorphisms between them cannot be extended to homeomorphisms of $\w^*$.

\begin{proof}[Proof of Theorem~\ref{thm:algebra}]
To prove the main assertion of the theorem, first note that, by Lemma~\ref{lem:thicksets}, it suffices to find an ultrafilter on $\thick$ that is also a weak $P$-filter. This follows directly from Theorem~\ref{thm:comeagerOKset} and the fact that $\thick$ is $G_\delta$. To see that $\thick$ is $G_\delta$, let
$$g(A) = \sup \set{n}{A \text{ contains an interval of length } n}$$
for every $A \sub \w$ (possibly $g(A) = \infty$), and apply Proposition~\ref{prop:functional}.

For $(2)$, let $L$ be a weak $P$-set and a minimal left ideal. If $q \in I$, then $p+q \in L$ because $L$ is a left ideal. Since $L$ is closed under $\s$ and $\s^{-1}$ by Lemma~\ref{lem:thicksets}, $\set{\s^n(q)}{n \in \N} \cap L = \0$. As $L$ is a weak $P$-set, $\closure{\set{\s^n(q)}{n \in \N}} \cap L = \0$. If $q \notin L$, then $p + q = \plim \s^n(q)$ is an element of $\closure{\set{\s^n(q)}{n \in \N}}$, so $p + q \notin L$.

For $(1)$, let $L$ be a minimal left ideal that is a weak $P$-set, and let $q \in L$ be idempotent. Since $q \in L$ and $L$ is a minimal left ideal, $q$ is minimal. Let $p$ be any idempotent other than $q$. If $p \in L$ then $p$ is $\leq$-minimal, hence $\leq_L$-minimal (see Proposition 1.36 in \cite{H&S}), so $q \not \leq_L p$. If $p \notin L$ then $p + q \notin L$ by $(1)$, in which case $q \not \leq_L p$. Thus $q$ is $\leq_L$-maximal.

For $(3)$, it suffices to note that some minimal left ideal is not a weak $P$-set. This is well-known and easy to prove: simply take some $p \in \w^*$ that is not in any minimal left ideal, and note that $\w^*+p = \closure{\w+p}$ contains a minimal left ideal.
\end{proof}

Further applications of Section~\ref{sec:weakpsets} and related ideas to the theory of semigroups would take us too far afield here, but these will be explored in a forthcoming sequel to this paper. We now move on to some different applications of the results in the preceding sections.

Say that a semifilter $\S$ has the \emph{Ramsey property} if for every $A \in \S$, if $A = \bigcup_{i \leq n}A_i$, then there is some $i \leq n$ with $A_i \in \S$. Such filters are also sometimes called \emph{co-ideals}, \emph{filterduals}, or \emph{superfilters} (see, e.g., Chapter 2 of \cite{Akn} for more on these). The following lemma is essentially due to Glasner (see \cite{Gls}).

\begin{lemma}\label{lem:glasner}
Let $\S$ be a semifilter with the Ramsey property. If $\F$ is an ultrafilter on $\S$ then $\F$ is an ultrafilter on $[\w]^{\w}$ (i.e., a free ultrafilter in the usual sense).
\end{lemma}
\begin{proof}
Let $\S$ be a semifilter with the Ramsey property and let $\F$ be an ultrafilter on $\S$. Suppose $\F$ is not an ultrafilter on $[\w]^{\w}$: then there is some infinite $A \sub \w$ such that $A \notin \F$ and $\w \setminus A \notin \F$. Since $\F$ is an ultrafilter on $\S$, neither $\{A\} \cup \F$ nor $\{\w \setminus A\} \cup \F$ is centered in $\S$. Thus there are $X,Y \in \F$ such that $X \cap A \notin \S$ and $Y \setminus A \notin \S$. However, $Z = X \cap Y \in \F$. Since $\S$ is closed under supersets, we cannot have $Z \cap A \in \S$ (otherwise $X \cap A \in \S$) or $Z \setminus A \in \S$ (otherwise $Y \setminus A \in \S$). This contradicts the Ramsey property in $\S$, since $Z \in \S$ and $Z = (Z \cap A) \cup (Z \setminus A)$.
\end{proof}

The non-parenthetical assertion of the following result was proved under the hypothesis $\mathrm{cov}(\text{meager}) = \continuum$ by Jana Fla\v{s}kov\'a in \cite{Fls}. Using Theorem~\ref{thm:d=c}, we can improve her hypothesis to $\dom = \continuum$.

\begin{proposition}\label{prop:coolppoints}
Suppose $\dom = \continuum$ (respectively, only assume \emph{ZFC}). If $\G$ is a $G_\delta$ semifilter with the Ramsey property, then there is a (weak) $P$-point in $\w^*$ with $p \sub \G$.
\end{proposition}
\begin{proof}
By Theorem~\ref{thm:d=c} (respectively, Theorem~\ref{thm:comeagerOKset}), there is an ultrafilter on $\G$ that is also a (weak) $P$-filter. By Lemma~\ref{lem:glasner}, this filter is a free ultrafilter in the usual sense.
\end{proof}

Examples of $G_\delta$ semifilters with the Ramsey property include:
\begin{itemize}
\item the infinite sets, in which case Proposition~\ref{prop:coolppoints} reduces to the classical theorem of Ketonen (resp., Kunen).
\item the sets containing arbitrarily long arithmetic sequences.
\item sets $A$ such that $\sum_{n \in A \setminus \{0\}}\frac{1}{n}$ diverges.
\item for a fixed enumeration $\set{e_n}{n \in \w}$ of the edges of $K_\w$ (the complete graph on $\w$), the sets $A \sub \w$ such that $\set{e_n}{n \in A}$ contains a copy of $K_n$ for every $n$.
\end{itemize}

\begin{remark}\label{rem:vdW}
For certain $\G$, the $P$-points exhibited in Proposition~\ref{prop:coolppoints} cannot be selective. For example, consider the partition of $\N$ into the intervals $I_0 = [0,0], I_1 = [1,2], I_2 = [3,6], I_3 = [7,14], I_4 = [15,30], \dots$ (in general, $I_n = [a_n,2a_n]$, where $a_n = \min (\w \setminus \bigcup_{m < n}I_m)$). If $A$ is any set containing a single point of each $I_n$, then $A$ does not contain arbitrarily long arithmetic sequences (it contains none of length $3$).
\end{remark}

Contrary to the main theme of this paper, the following corollary of Proposition~\ref{prop:coolppoints} demonstrates a nontrivial way in which $G_\delta$ semifilters do not necessarily behave like $[\w]^{\w}$.

\begin{proposition}
It is consistent to have $P$-ultrafilters on some $G_\delta$ semifilters but not others.
\end{proposition}
\begin{proof}
In Section XVIII.4 of \cite{Shl}, Shelah proves that it is consistent to have a selective ultrafilter and no $P$-points that fail to be selective. Combining this with Remark~\ref{rem:vdW}, it is consistent to have a $P$-ultrafilter on $[\w]^\w$, but none on the semifilter of sets containing arbitrarily long arithmetic sequences.
\end{proof}

Proposition~\ref{prop:coolppoints} suggests a strengthening of the Ramsey property. Recall that $p \in \w^*$ is a $P$-point if and only if whenever there is a sequence $A_n$ of sets not in $p$, then there is a set in $p$ that meets each of them finitely. Let us say that a semifilter $\S$ is \emph{countably Ramsey} if, for every $A \in \S$, if $A = \bigcup_{n \in \w}A_n$, then either
\begin{enumerate}
\item $A_n \in \S$ for some $n$; or
\item there is a set $B \sub A$ such that $B \in \S$ and, for each $n$, $B \cap A_n$ is finite.
\end{enumerate}

Proposition~\ref{prop:coolppoints} suggests that $G_\delta$ semifilters might possess this property. We prove this next. Notice that the proof does not require any $P$-points and is carried out in ZFC; in fact, it seems to be merely another combinatorial expression of the fact that $\pseudo > \aleph_0$.

\begin{proposition}\label{prop:countableindivisibility}
Suppose $\G$ is a $G_\delta$ semifilter with the Ramsey property. Then $\G$ is countably Ramsey. 
\end{proposition}
\begin{proof}
Let $A \in \G$ and $A = \bigcup_{n \in \w}A_n$. Suppose that none of the $A_n$ are in $\G$. Putting $B_n = A \setminus \bigcup_{m \leq n}A_{k_m}$, the Ramsey property guarantees that $B_n \in \G$. Furthermore, the $B_n$ form a chain in $\G$. By Theorem~\ref{thm:M&S}, $\pseudo_\G = \pseudo > \card{B}$, so that $\set{B_n}{n \in \w}$ has a lower bound $B$ in $\G$. This $B$ must satisfy $(2)$ from the definition of countably Ramsey.
\end{proof}

Versions of this property have already been defined and studied, e.g., rainbow Ramsey properties (\cite{FMO}) or the canonical van der Waerden theorem (\cite{E&G}). In contrast with these two variants, however, the chief interest of the countably Ramsey property is that it makes sense to ask of any semifilter whether it is countably Ramsey.

\begin{question}
Which other semifilters are countably Ramsey?
\end{question}

Clearly every countably Ramsey semifilter has the Ramsey property, but the converse implication fails. The easiest way to see this is to consider an ultrafilter that is not a $P$-point. We conclude with a more interesting example: the IP sets. Recall that a set is \emph{IP} iff it satisfies the conclusion of Hindman's finite-sums Theorem. Specifically, $A \sub \w$ is an IP set iff there is some $\set{a_i}{i \in \w} \sub A$ such that for any finite $F \sub \w$ there is some $j$ with $\sum_{i \in F}a_i = a_j$.

\begin{proposition}\label{prop:notIPorcentral}
The semifilter of IP sets is not countably Ramsey.
\end{proposition}
\begin{proof}
For each $n \in \w$, let $A_n = \set{2^n(2k+1)}{k \in \w}$. $\bigcup_{n \in \w}A_n = \w \setminus \{0\}$, and it is easily checked that none of the $A_n$ is an IP set.

Suppose $B \sub \w$, $B$ is an IP set, and $B \cap A_n$ is finite for every $n$. Let $B_0 = \set{b_n}{n \in \w}$ witness the fact that $B$ is an IP set. Without loss of generality, we may assume that our enumeration of $B_0$ is increasing.

Fix $n$ such that $B_0 \cap A_n \neq \0$, and let $b_k \in B_0 \cap A_n$. Since $B_0 \cap A_m$ is finite for all $m \leq n$, there is some $L$ such that if $\ell \geq L$ then $b_\ell \notin \bigcup_{m \leq n}A_m$. In other words, if $\ell \geq L$, $b_\ell$ is even modulo $2^n$. But $b_k$ is odd modulo $2^n$, so $b_\ell + b_k \in A_n$. Since $b_\ell + b_k = b_j$ for some $j > \ell \geq L$, this contradicts our choice of $L$ and shows that $B$ cannot be an IP set.

Thus we have $\w = \{0\} \cup \bigcup_{n \in \w}A_n$, none of the $A_n$ is an IP set, and there is no IP set meeting every $A_n$ finitely.
\end{proof}

The following corollary shows that Proposition~\ref{prop:countableindivisibility} does not extend to co-meager semifilters:

\begin{corollary}\label{cor:IPsets}
There is a co-meager semifilter with the Ramsey property that fails to be countably Ramsey.
\end{corollary}
\begin{proof}
Let $\S$ denote the semifilter $\S$ of IP sets. It is easily seen (e.g., by the methods in Appendix C of \cite{Kec}) that $\S$ is analytic, and hence has the property of Baire (see, e.g., Exercise 8.50 in \cite{Kec}). By Proposition~\ref{prop:bairemeager}, $\S$ is either meager or co-meager. But $\S$ cannot be meager since, by Lemma~\ref{lem:glasner}, it contains an ultrafilter (and no ultrafilter is meager; see, e.g., Theorem 29.5 in \cite{Kec}).
\end{proof}

\section*{Acknowledgement}

The second author would like to thank the members of the Prague Set Theory seminar for fruitful discussions of the topic. In particular, the main idea of the proof of Proposition~\ref{prop:functional} came out of these discussions. Work on the paper was partially completed while the second author was visiting the Institute for Mathematical Sciences, National University of Singapore in 2015.

\end{document}